\documentclass[final,5p,times,twocolumn]{elsarticle}

\usepackage{amsmath, amssymb, amsthm, amsfonts}
\usepackage{graphicx}
\usepackage{hyperref}
\usepackage[british]{babel}
\usepackage{geometry}

\journal{arXiv preprint}

\newtheorem{theorem}{Theorem}[section]
\newtheorem{definition}[theorem]{Definition}

\newtheorem{remark}[theorem]{Remark}

\newcommand{\R}{\mathbb{R}}
\newcommand{\Fpw}{\mathcal{F}_{\phi,\omega}}       
\newcommand{\Jphi}{|J_\phi(x)|} 
\newcommand{\Lpw}{L^2_{\phi,\omega}}
\newcommand{\Qxi}{Q(\xi)}                            
\newcommand{\Pphi}{P(\phi(x))}                       
\makeatletter
\def\ps@pprintTitle{%
 \let\@oddhead\@empty
 \let\@evenhead\@empty
 \def\@oddfoot{\centerline{\thepage}\hfill\textit{arXiv preprint}}%
 \def\@evenfoot{\centerline{\thepage}\hfill\textit{arXiv preprint}}%
}
\makeatother
\begin{document}

\begin{frontmatter}

\title{Weighted fractional ultrahyperbolic diffusion on geometrically deformed domains}

\author{Gustavo A. Dorrego}
\ead{gadorrego@exa.unne.edu.ar} 
\address{Department of Mathematics, Faculty of Exact and Natural Sciences and Surveying, Universidad Nacional del Nordeste, Corrientes, Argentina.} 

\begin{abstract}
Standard fractional models on manifolds often conflate geometric anisotropy with medium heterogeneity. In this Letter, we overcome this rigidity by deriving the fundamental solution for a weighted space-time fractional ultrahyperbolic operator, denoted by $(-\Box_{\phi,\omega})^{\beta}$. Using a novel spectral approach based on the Weighted Fourier Transform, we explicitly \textbf{decouple the medium density from the geometric deformation}. A crucial finding is the emergence of a \textbf{geometry-independent drift mechanism} driven purely by the inhomogeneity of the medium. The Green's function is obtained in closed form via the Fox H-function, providing a unified and computable framework for anomalous transport in complex, structurally deformed media.
\end{abstract}

\begin{keyword}
Fractional ultrahyperbolic operator \sep Weighted Fourier transform \sep Heterogeneous media \sep Fox H-function \sep Geometric drift
\MSC[2020] 26A33 \sep 35R11 \sep 42B10 \sep 33C60
\end{keyword}

\end{frontmatter}

\section{Introduction}

Fractional differential equations have proven to be indispensable for modelling non-local phenomena in fractal and viscoelastic media \cite{Kilbas2006, Mainardi2010}. However, standard fractional operators, including the ultrahyperbolic class studied in \cite{Dorrego2016, Dorrego2021}, are typically defined on Euclidean spaces with flat metrics. This restricts their applicability in systems exhibiting geometric deformations or variable densities, where the direction of anomalous transport is coupled to the inhomogeneity of the medium. To address this, the framework of fractional calculus with respect to a function \cite{Almeida2017, Fernandez2024} offers a pathway to incorporate geometric modulation directly into the operator structure.

In this letter, we construct the fundamental solution for the weighted space-time fractional ultrahyperbolic equation:
\begin{equation} \label{eq:intro_problem}
    {}^{H}\mathcal{D}_{t,\gamma,\rho}^{\mu,\nu} u(x,t) + c^2 (-\Box_{\phi,\omega})^{\beta} u(x,t) = 0, \quad x \in \mathbb{R}^n, t>0,
\end{equation}
where the spatial operator acts on a manifold deformed by a diffeomorphism $\phi$. Our approach relies on the \textit{weighted Fourier transform} $\mathcal{F}_{\phi,\omega}$, whose spectral properties have been rigorously established in \cite{DorregoFourier2025}. By leveraging this spectral framework, we overcome the limitations of standard Euclidean models and obtain the Green's function in a closed form involving the \textbf{Fox H-function}.  Crucially, we reveal that the decoupling of density and geometry induces a natural drift term,  and show that the system's evolution is governed by a \textbf{generalised ultrahyperbolic distance $\Pphi$}, providing a robust tool for modelling anomalous transport in anisotropic inhomogeneous media.

The remainder of this Letter is organised as follows: Section 2 outlines the theoretical background of weighted fractional operators; Section 3 presents the spectral definition of the operator and the emergence of the drift; and Section 4 derives the fundamental solution.

\section{Preliminaries}

We adopt the generalised framework established in \cite{FernandezFahad2022_Conj}. Let $\gamma \in C^1[0, \infty)$ be an increasing function with $\gamma'>0$ and $\rho(t)$ a positive weight. The \textit{weighted Hilfer fractional derivative} ${}^{H}\mathcal{D}_{t,\gamma,\rho}^{\alpha,\beta}$ (for $m-1 < \alpha \le m$) is defined via weighted fractional integrals as in \cite{FernandezFahad2022_Conj}. Its generalised Laplace transform is given by:
\begin{equation} \label{eq:Laplace_Hilfer}
    \mathcal{L}_{\gamma,\rho} \left[ {}^{H}_{0+}\mathcal{D}_{t,\gamma,\rho}^{\alpha,\beta} \psi \right](z) = z^\alpha \tilde{\psi}(z) - \sum_{k=0}^{m-1} c_k z^{m(1-\beta) + \alpha\beta - k - 1},
\end{equation}
where the constants $c_k$ depend on the initial values of the weighted integrals of $\psi$.

For the spatial spectral analysis, we consider a diffeomorphism $\phi: \mathbb{R}^n \to \mathbb{R}^n$ with Jacobian $\Jphi$ and a non-vanishing weight $\omega(x)$. We recall the core spectral definitions from \cite{DorregoFourier2025}:

\begin{definition}[Weighted Fourier Transform and Inversion \cite{DorregoFourier2025}]\label{def:fourier_grad}
    The Weighted Fourier Transform $\Fpw$ of a function $f \in \Lpw(\R^n)$ is defined as:
    \begin{equation} \label{eq:def_Rn}
        [\Fpw f](\xi) = \frac{1}{(2\pi)^{n/2}} \int_{\R^n} e^{-i \xi \cdot \phi(x)} \omega(x) f(x) \Jphi \, dx.
    \end{equation}
    The corresponding \textbf{Inverse Weighted Fourier Transform} is given by:
    \begin{equation} \label{eq:inv_def_Rn}
        [\mathcal{F}^{-1}_{\phi,\omega} g](x) = \frac{1}{\omega(x)(2\pi)^{n/2}} \int_{\R^n} e^{i \xi \cdot \phi(x)} g(\xi) \, d\xi.
    \end{equation}
    Associated with this structure is the \textbf{weighted deformed gradient} $\nabla_{\phi,\omega}$, defined by the correspondence $\Fpw[\nabla_{\phi,\omega} f](\xi) = i\xi \Fpw[f](\xi)$, or explicitly in the spatial domain as:
    \begin{equation}
        \nabla_{\phi,\omega} f(x) := \frac{1}{\omega(x)} \left( [D\phi(x)]^{-T} \nabla \right) \big( \omega(x) f(x) \big).
    \end{equation} 
\end{definition}

\section{The Weighted Ultrahyperbolic Operator}

To establish a general framework for anomalous diffusion in heterogeneous media, we proceed to define the spatial operator directly via the spectral theory developed in \cite{DorregoFourier2025}. Unlike classical Riemannian approaches that rigidly couple the measure to the metric, we allow the weight function $\omega(x)$ (representing medium density) to be independent of the geometric deformation $\phi(x)$.

\subsection{Induced Geometry and Causal Cones}

The algebraic structure of the operator is intrinsically linked to the geometry induced by the diffeomorphism $\phi$.

\begin{definition}[Deformed Ultrahyperbolic Distance]
We define the \textbf{generalized ultrahyperbolic distance} $P_{\phi}: \R^n \to \R$ associated with the diffeomorphism $\phi$ as:
\begin{equation} \label{eq:deformed_distance}
    P(\phi(x)) := \sum_{j=1}^{p} [\phi_j(x)]^2 - \sum_{k=p+1}^{n} [\phi_k(x)]^2.
\end{equation}
This function generalises the classical Lorentzian/ultrahyperbolic distance to the space endowed with the geometry induced by $\phi$.
\end{definition}

\begin{remark}[The Generalised Characteristic Cone]
The level set equation $P(\phi(x)) = 0$ defines the \textit{``Deformed Light Cone''} or characteristic manifold. Unlike the classical case, where the cone is a straight quadric hypersurface, here the surface curves following the flow lines of $\phi$. The singularities of the fundamental solution will propagate along this manifold.
\end{remark}

\begin{theorem}[Spectral Identity for the Density-Compensated Distance] \label{thm:spectral_identity}
Given the definition of the Weighted Fourier Transform $\mathcal{F}_{\phi,\omega}$ which incorporates the Jacobian determinant in its measure, the transform of the generalized distance scaled by the inverse density is given by:
\begin{equation}
\begin{split}
    \mathcal{F}_{\phi,\omega} &\left[ \frac{1}{\omega(x)} (\Pphi \pm i0)^\lambda \right] (\xi) \\
    &= \frac{e^{\mp i \frac{q \pi}{2}} 2^{2\lambda+n/2} \Gamma(\lambda + n/2)}{\Gamma(-\lambda)} (\Qxi \mp i0)^{-\lambda - n/2}.
\end{split}
\end{equation}
\end{theorem}

\begin{proof}
Let $f(x) = \frac{1}{\omega(x)} (\Pphi \pm i0)^\lambda$. Applying the definition \eqref{eq:def_Rn}:
\begin{equation}
\begin{split}
    [\mathcal{F}_{\phi,\omega} f](\xi) &= \frac{1}{(2\pi)^{n/2}} \int_{\R^n} e^{-i \xi \cdot \phi(x)} \omega(x) \\
    &\quad \times \left[ \frac{1}{\omega(x)} (\Pphi \pm i0)^\lambda \right] \Jphi \, dx.
\end{split}
\end{equation}
The weight $\omega(x)$ cancels explicitly:
\begin{equation}
    [\mathcal{F}_{\phi,\omega} f](\xi) = \frac{1}{(2\pi)^{n/2}} \int_{\R^n} e^{-i \xi \cdot \phi(x)} (\Pphi \pm i0)^\lambda \Jphi \, dx.
\end{equation}
We apply the global change of variables $y = \phi(x)$. Consequently, the volume element transforms as $dy = \Jphi dx$. The integral becomes:
\begin{equation}
    \mathcal{I}(\xi) = \frac{1}{(2\pi)^{n/2}} \int_{\R^n} (P(y) \pm i0)^\lambda e^{-i \xi \cdot y} \, dy.
\end{equation}
This is exactly the standard Fourier transform of the generalized quadratic form. Invoking the result by Gelfand and Shilov \cite{Gelfand1964}, we obtain the spectral symbol $(\Qxi \mp i0)^{-\lambda - n/2}$ multiplied by the corresponding constant.
\end{proof}
\subsection{Spectral Definition of the Operator}

We now introduce the fractional operator as a pseudo-differential operator characterized by its symbol in the weighted frequency domain.

\begin{definition}[Weighted Fractional Ultrahyperbolic Operator] \label{def:weighted_operator}
Let $f \in \mathcal{S}_{\phi,\omega}(\mathbb{R}^{n})$ be a function in the generalized Schwartz space defined in \cite{DorregoFourier2025}. We define the fractional ultrahyperbolic operator $(-\Box_{\phi,\omega})^{\beta}$ of order $\beta > 0$ by:
\begin{equation} \label{eq:operator_def}
    (-\Box_{\phi,\omega})^{\beta}f(x) := \mathcal{F}_{\phi,\omega}^{-1} \left[ (Q(\xi) \mp i0)^{\beta} \mathcal{F}_{\phi,\omega}[f](\xi) \right](x),
\end{equation}
where $\mathcal{F}_{\phi,\omega}$ is the weighted Fourier transform \eqref{eq:def_Rn} and its inverse is given by Theorem 3.3 in \cite{DorregoFourier2025}. The symbol $(Q(\xi) \mp i0)^{\beta}$ corresponds to the causal/anti-causal regularisation of the quadratic form associated with the ultrahyperbolic metric \cite{Samko1993}.
\end{definition}

\begin{remark}[Independence of Density and Geometry]
\label{rem:independence}
A key feature of Definition \ref{def:weighted_operator} is its flexibility. Unlike the Laplace-Beltrami operator on Riemannian manifolds, where the measure is strictly tied to the metric via $\omega(x) = |J_\phi(x)|$, our spectral definition allows the weight function $\omega(x)$ to be chosen freely. This is crucial for modeling complex materials where the particle distribution ($\omega$) does not necessarily follow the structural anisotropy ($\phi$).
\end{remark}

\subsection{Operational Decomposition and the Induced Drift}
\label{subsec:drift}

To understand the interaction between the geometric deformation and the medium density, it is essential to examine the structure of $\Box_{\phi,\omega}$ in the local limit $\beta=1$. In this regime, the operator is defined through a weighted divergence, and its action can be decomposed to reveal its underlying advective nature.

Let $g^{ij}(x)$ denote the coefficients of the inverse metric tensor induced by the diffeomorphism $\phi$, given by $g^{ij} = ([D\phi]^{-1} [D\phi]^{-T})_{ij}$. The operator acts on a scalar field $u$ as:
\begin{equation} \label{eq:operator_analysis}
    \Box_{\phi,\omega} u = \frac{1}{\omega(x)} \sum_{i,j=1}^{n} \frac{\partial}{\partial x_i} \left( \omega(x) g^{ij}(x) \frac{\partial u}{\partial x_j} \right).
\end{equation}

To decouple the effects of the mapping and the weight, we introduce the \textit{relative density ratio}:
\begin{equation}
    \sigma(x) = \frac{\omega(x)}{|J_\phi(x)|},
\end{equation}
where $|J_\phi(x)|$ is the Jacobian determinant of $\phi$. It is important to emphasize that $\sigma(x)$ is an analytical tool rather than a functional restriction; $\omega$ and $\phi$ remain independent degrees of freedom. By applying the product rule to \eqref{eq:operator_analysis}, we obtain the following operational decomposition:
\begin{equation}
    \Box_{\phi,\omega} u = \Delta_{\phi} u + \sum_{j=1}^n \mathcal{V}_j(x) \frac{\partial u}{\partial x_j}.
\end{equation}

Here, $\Delta_{\phi}$ is the \textbf{purely geometric operator} associated with the Riemannian volume measure $|J_\phi|dx$:
\begin{equation}
    \Delta_{\phi} u := \frac{1}{|J_\phi(x)|} \sum_{i,j=1}^{n} \frac{\partial}{\partial x_i} \left( |J_\phi(x)| g^{ij}(x) \frac{\partial u}{\partial x_j} \right).
\end{equation}

This term accounts for the diffusion dictated by the deformation $\phi$, while the second term—a \textbf{first-order differential operator}—characterizes a transport process induced by the non-equilibrium weight. The coefficients of this drift field $\mathcal{V}_j(x)$ are explicitly given by:
\begin{equation}
    \mathcal{V}_j(x) = \sum_{i=1}^n g^{ij}(x) \frac{\partial}{\partial x_i} \ln \sigma(x).
\end{equation}

In the context of anomalous transport, this decomposition signifies a symmetry breaking. While $\Delta_\phi$ models locally anisotropic diffusion, the drift term $\mathcal{V}_j$ induces a net translation towards regions where the density $\omega(x)$ dominates the geometric volume element. This structure is formally analogous to Fokker-Planck dynamics, confirming that the decoupling of density and geometry naturally incorporates an effective bias into the model.

\section{Fundamental Solution of the Cauchy Problem}

Before stating the main result, we emphasize that the initial conditions are formulated in terms of the weighted fractional integral $I_{0+,\gamma,\rho}^{(1-\nu)(1-\mu)}$. This specific order is intrinsic to the definition of the weighted Hilfer derivative ${}^{H}\mathcal{D}_{t,\gamma,\rho}^{\mu,\nu}$. Furthermore, to ensure consistency with the spectral geometry, the spatial excitation is given by the weighted Dirac delta $\delta_{\phi,\omega}(x)$. This distribution is defined by the sifting property $\int_{\mathbb{R}^n} f(x) \delta_{\phi,\omega}(x) \omega(x) dx = f(0)$, acting as the identity element for the weighted convolution.

The operator decomposition established in Subsection \ref{subsec:drift} implies that the propagator of the system must simultaneously resolve the anisotropic diffusion and the advective bias. Analytically, this means the fundamental solution will not exhibit the classical radial symmetry of the ultrahyperbolic case. Instead, the interplay between $\Delta_\phi$ and the drift field $\mathcal{V}_j(x)$ dictates a deformation of the characteristic manifolds, which is captured by the integral representation in terms of Fox $H$-functions.

The main result regarding the solution of the Cauchy problem is presented in the following theorem:

\begin{theorem}[Fundamental Solution] \label{thm:main_sol}
    Consider the weighted space-time fractional ultra-hyperbolic problem:
    \begin{equation} \label{eq:problem_statement}
       \left\{
       \begin{array}{l}
          {}^{H}\mathcal{D}_{t,\gamma,\rho}^{\mu,\nu} u(x,t) + c^2 (-\Box_{\phi,\omega})^{\beta} u(x,t) = 0, \quad x \in \mathbb{R}^n, t>0, \\
          \\
          I_{0+,\gamma,\rho}^{(1-\nu)(1-\mu)} u(x,0^+) = \delta_{\phi,\omega}(x), \quad \lim_{|x|\to\infty} u(x,t) = 0.
       \end{array}
       \right.
    \end{equation}
    where $\mu \in (0,1]$, $\nu \in [0,1]$, $\beta > 0$, and $\delta_{\phi,\omega}(x)$ denotes the weighted Dirac delta distribution.
    
    The fundamental solution $\mathcal{G}(x,t)$ is given by the following Mellin-Barnes integral representation:
    \begin{equation} \label{eq:fund_sol_mellin}
        \mathcal{G}(x,t) =\Lambda(x,t)\frac{1}{2\pi i}\int_{L}{\frac{\Gamma(n/2-\beta s)\Gamma(s)\Gamma(1-s)}{\Gamma(\beta s)\Gamma(\nu-\mu s)}\mathcal{Z}(x,t)^{-s}}ds.
    \end{equation}
    where the pre-factor $\Lambda(x,t)$ is defined as:
    \begin{equation}
        \Lambda(x,t) = \frac{e^{\pm i\pi q/2}\rho(0^+)(\gamma(t))^{\nu-1}}{\rho(t)\pi^{n/2}} \frac{1}{\omega(x)} \frac{1}{\left(P(\phi(x))\pm i0\right)^{n/2}}, 
    \end{equation}
    and the scaling argument is $\mathcal{Z}(x,t) = 4^{\beta}c^2 (\gamma(t))^{\mu}(P(\phi(x))\pm i0)^{-\beta}$. Equivalently, in terms of the Fox H-function:
    \begin{equation} \label{eq:fox_solution_compact}
        \mathcal{G}(x,t) = \Lambda(x,t) H_{3,2}^{1,2} \left[ \mathcal{Z} \left| \begin{matrix} (1 - \frac{n}{2}, \beta), (0, 1), (0, \beta) \\ (0, 1), (1-\nu, \mu) \end{matrix} \right. \right].
    \end{equation}
\end{theorem}

\begin{proof}
To find the fundamental solution, we apply the Weighted Fourier Transform defined in \eqref{eq:def_Rn} with respect to the spatial variable $x$. Let $\widehat{u}(\xi, t) = \mathcal{F}_{\phi,\omega}[u(x,t)](\xi)$.
Applying the transform to the problem \eqref{eq:problem_statement} and utilizing the spectral property that the transform of $(-\Box_{\phi,\omega})^\beta$ corresponds to multiplication by the symbol $(Q(\xi) \mp i0)^\beta$, we obtain the following fractional ordinary differential equation in time:
\begin{equation} \label{eq:ode_time}
    {}^{H}\mathcal{D}_{t,\gamma,\rho}^{\mu,\nu} \widehat{u}(\xi, t) + c^2 (Q(\xi) \mp i0)^\beta \widehat{u}(\xi, t) = 0.
\end{equation}
Applying the generalized Laplace transform $\mathcal{L}_{\gamma,\rho}$ and solving the algebraic equation leads to the solution in the frequency domain:
\begin{equation} \label{eq:sol_freq_ML}
    \widehat{u}(\xi, t) = \frac{\rho(0^+) (\gamma(t))^{\nu-1}}{\rho(t)} E_{\mu, \nu} \left( -c^2 (\gamma(t))^\mu (Q(\xi) \mp i0)^{\beta} \right).
\end{equation}
Substituting the Mittag-Leffler function by its Mellin-Barnes integral representation and applying the Inverse Weighted Fourier Transform $\mathcal{F}^{-1}_{\phi,\omega}$:
\begin{equation}
\begin{split}
    u(x,t) &= \frac{\rho(0^+) (\gamma(t))^{\nu-1}}{\rho(t)} \frac{1}{2\pi i} \\
    &\quad \times \int_{\mathcal{L}} \frac{\Gamma(s)\Gamma(1-s)}{\Gamma(\nu - \mu s)} (c^2 (\gamma(t))^\mu)^{-s} \underbrace{\mathcal{F}^{-1}_{\phi,\omega} \left[ (Q(\xi) \mp i0)^{-\beta s} \right]}_{\mathcal{I}(x, s)} \, ds.
\end{split}
\end{equation}
To evaluate the inner term $\mathcal{I}(x, s)$, we use the definition of the inverse transform given in Eq. \eqref{eq:inv_def_Rn}. The factor $1/\omega(x)$ appears explicitly from the definition. The integral component is:
\begin{equation}
    \frac{1}{(2\pi)^{n/2}} \int_{\mathbb{R}^n} e^{i \xi \cdot \phi(x)} (Q(\xi) \mp i0)^{-\beta s} \, d\xi.
\end{equation}
This integral corresponds exactly to the standard inverse Fourier transform of the symbol $(Q(\xi) \mp i0)^{-\beta s}$ evaluated at the spatial coordinate $y = \phi(x)$. Invoking the classical result for the Riesz kernel \cite{Gelfand1964}, we obtain:
\begin{equation}\label{eq:Riesz_Inverse}
\begin{split}
    \mathcal{F}^{-1}_{\phi,\omega} \left[ (Q(\xi) \mp i0)^{-\beta s} \right] &= \frac{1}{\omega(x)} \frac{e^{\mp i \frac{\pi}{2}(2\beta s - n)} \Gamma(n/2 - \beta s)}{\pi^{n/2} 2^{2\beta s} \Gamma(\beta s)} \\
    &\quad \times (P(\phi(x)) \pm i0)^{\beta s - n/2}.
\end{split}
\end{equation}
Substituting \eqref{eq:Riesz_Inverse} back into the integral representation and rearranging the Gamma terms, the expression matches the definition of the Fox H-function, yielding the explicit representation \eqref{eq:fox_solution_compact}.
\end{proof}

\begin{remark}[Convergence and Existence Conditions]\label{rem:convergence}
    Following the standard theory of the H-function (see \cite{Mathai2010}), we analyze the conditions for the existence and convergence of the fundamental solution obtained in Theorem \ref{thm:main_sol}.
    
    First, for the Mellin-Barnes contour $L$ to exist, the poles of the gamma functions in the numerator of the integrand must be separated. In our case, the poles of $\Gamma(s)$ are located at $s = -k$ ($k \in \mathbb{N}_0$), whereas the poles of $\Gamma(n/2 - \beta s)$ and $\Gamma(1-s)$ lie on the positive real axis (given that $\beta > 0$ and $n \ge 1$). Therefore, the poles are strictly separated, ensuring the contour $L$ is well-defined.

    Second, the absolute convergence of the integral is governed by the parameter $a^*$, defined as:
    \begin{equation}
        a^* = \sum_{j=1}^n \alpha_j - \sum_{j=n+1}^p \alpha_j + \sum_{j=1}^m \beta_j - \sum_{j=m+1}^q \beta_j.
    \end{equation}
    Substituting the parameters from our solution $H_{3,2}^{1,2}[\mathcal{Z}]$:
    \begin{equation}
        a^* = (\beta + 1) - \beta + 1 - \mu = 2 - \mu.
    \end{equation}
    The condition for absolute convergence, $a^* > 0$, implies $\mu < 2$. This constraint is physically consistent, as it covers the regimes of fractional diffusion ($0 < \mu \le 1$) and fractional diffusion-wave ($1 < \mu < 2$). The limiting case $\mu \to 2$ corresponds to the classical ultrahyperbolic wave equation, where the sector of analyticity vanishes ($a^* = 0$), consistent with the emergence of singular wavefronts on the characteristic cone.
    
    Additionally, the asymptotic behavior for large arguments is determined by $\mu^* = \sum_{j=1}^q \beta_j - \sum_{j=1}^p \alpha_j = \mu - 2\beta$. Even in cases where $\mu^* < 0$ (common in sub-diffusion), the solution remains convergent for all $\mathcal{Z} \neq 0$ as long as the condition $a^* > 0$ holds.
\end{remark}

\begin{remark}[Asymptotic Behavior and Heavy Tails]
\label{rem:asymptotic}
To elucidate the nature of the transport, we analyze the asymptotic behavior of the fundamental solution $\mathcal{G}(x,t)$ for large spatial distances (far-field limit). We consider the behavior of the Fox $H$-function for small arguments of $|\mathcal{Z}(x,t)|$, which corresponds to the limit $|\Pphi| \to \infty$.

Using the power series expansion of the $H$-function near zero, we note that the term corresponding to $k=0$ vanishes due to the presence of the factor $1/\Gamma(\beta s)$ in the Mellin-Barnes integrand (since $\Gamma(0) \to \infty$ in the denominator). Consequently, the series starts from $k=1$:
\begin{equation}
\begin{split}
    \mathcal{G}(x,t) &\sim \frac{C_{\text{pre}}}{\omega(x)(\Pphi \pm i0)^{n/2}} \\
    &\quad \times \sum_{k=1}^{\infty} \frac{(-1)^{k}}{k!} \mathcal{A}_k \left( \frac{4^{\beta}c^{2}(\gamma(t))^{\mu}}{(\Pphi \pm i0)^{\beta}} \right)^{k},
\end{split}
\label{eq:asymp_expansion}
\end{equation}
where $\mathcal{A}_k$ represents the quotient of Gamma functions associated with the expansion. Considering the leading term ($k=1$), the solution exhibits an algebraic decay of the form:
\begin{equation}
    \mathcal{G}(x,t) \propto \frac{1}{\omega(x)} \frac{1}{[\Pphi \pm i0]^{n/2+\beta}}.
    \label{eq:power_law_decay}
\end{equation}

This result offers a significant analytical interpretation:
\begin{enumerate}
    \item \textbf{Anomalous Geometry and Drift Manifestation:} The decay is not governed by the Euclidean distance $|x|$, but by the generalized ultrahyperbolic distance $\Pphi$ modulated by the local density $\omega(x)$. Crucially, the factor $1/\omega(x)$ acts as a spatial renormalization: in high-density regions, the propagator's amplitude is suppressed, effectively pushing the probability mass toward lower-density zones. This is the macroscopic manifestation of the drift field $\mathcal{V}_j(x)$ derived in Subsection 4.1.
    
    \item \textbf{Heavy Tails (Super-diffusion):} The power-law decay $|x|^{-(n+2\beta)}$ (assuming $\phi(x) \sim x$ asymptotically) confirms that the process exhibits non-local interactions. This contrasts with the exponential decay of Gaussian diffusion. The algebraic tail indicates a higher probability of large jumps, modeling super-diffusive transport along the geodesics induced by the deformation $\phi$.
\end{enumerate}
\end{remark}
\section{Relevant Particular Cases}

The analytical strength of the general solution obtained in Theorem \ref{thm:main_sol} lies in its capacity to recover classical results and to model anomalous dynamics in complex media through the judicious choice of the diffeomorphism $\phi$ and the weight $\omega$.

\subsection{Recovery of the Classical Case}
Consider the identity transformation $\phi(x) = x$ and the unit weight $\omega(x) = 1$. In this configuration, the induced metric is the standard Minkowski/Euclidean metric $g^{ij} = \delta_{ij}\eta_{ii}$ and the relative density ratio is $\sigma(x) = 1$. Consequently, the drift coefficients $\mathcal{V}_j$ vanish identically, and the fundamental solution $\mathcal{G}(x,t)$ reduces to the standard fractional ultrahyperbolic propagator (cf. \cite{Samko1993}). This confirms the consistency of our framework as a proper generalization of the homogeneous theory.

\subsection{Case Study: Exponentially Graded Media}
\label{subsec:exponential}

To illustrate the interplay between geometry and density, we analyze a domain characterized by strong directional anisotropy. We introduce the exponential deformation map\\ $\phi:\mathbb{R}^{n}\rightarrow\mathbb{R}^{n}$ defined component-wise by: 
\begin{equation}
\label{eq:exp_map}
\phi_{k}(x_{k})=\frac{e^{\lambda_{k}x_{k}}-1}{\lambda_{k}}, \quad k=1,\dots,n,
\end{equation} where $\lambda_{k}$ represents the degree of spatial stretching. 

In this scenario, we \textit{specifically choose} to prescribe the medium density in geometric equilibrium with the deformation by setting:
\begin{equation}
\label{eq:weight_exp}
\omega(x) = |\det J_\phi(x)| = \exp\left(\sum_{k=1}^{n}\lambda_{k}x_{k}\right).
\end{equation}
By setting $\sigma(x) = 1$, we intentionally nullify the advective drift field $\mathcal{V}_j$ derived in Subsection \ref{subsec:drift}. This choice is motivated by the need to \textbf{isolate and visualize the purely metric effects} of the transformation, providing a clear benchmark for the theory. This demonstrates that even in the absence of an explicit drift, the stretching of the underlying space is sufficient to break the radial symmetry of the transport process.

\begin{figure}[h!]
    \centering
    \includegraphics[width=\linewidth]{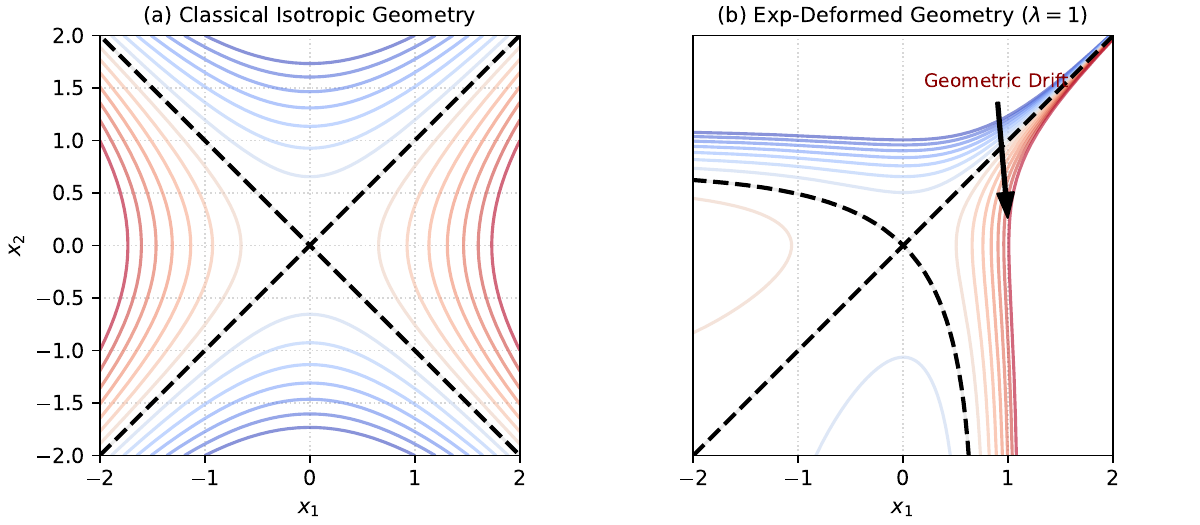}
    \caption{Impact of geometric deformation on the characteristic manifolds. 
    \textbf{(a)} Classical ultrahyperbolic geometry ($\lambda \to 0$), showing symmetric hyperbolic characteristics and linear light cones. 
    \textbf{(b)} Deformed geometry under the exponential map with $\lambda=1$. The inhomogeneity induces a curvature in the characteristic curves, breaking the spatial symmetry. This visualizes the structural origin of the bias: although the drift field $\mathcal{V}_j$ is nullified by the choice of $\omega$, the diffusion is 'guided' by the exponential stretching of the underlying space.}
    \label{fig:deformation}
\end{figure}

Analytically, this represents a material where the effective metric becomes exponentially singular. Under this configuration, the fundamental solution describes a fractional diffusion process where the wavefront is exponentially suppressed in the direction of increasing density, despite the absence of an explicit advective field. Substituting \eqref{eq:exp_map} and \eqref{eq:weight_exp} into Theorem \ref{thm:main_sol}, the propagator reveals how the non-local tails adapt to the graded profile of the medium.

\section{Discussion and Conclusions}

In this work, we have derived the fundamental solution for the space-time fractional ultrahyperbolic equation on geometrically deformed domains with non-uniform weights. By leveraging the spectral theory of the weighted Fourier transform \cite{DorregoFourier2025}, we constructed the operator $(-\Box_{\phi,\omega})^{\beta}$ as a pseudo-differential operator in the sense of tempered distributions. This approach successfully overcomes the geometric density constraint inherent in classical Riemannian models, allowing for independent control over the medium's anisotropy (via $\phi$) and its mass density (via $\omega$).

The analytical and physical implications of this framework are three-fold:
\begin{enumerate}
    \item \textbf{Deformed Causality:} The propagation of the fractional wavefront is no longer constrained by Euclidean distances but evolves along the characteristic manifolds defined by the generalized metric $P(\phi(x))=0$.
    \item \textbf{Emergent Drift Mechanism:} We have formally demonstrated that any mismatch between the medium's weight and the geometric volume element ($\omega \neq |J_\phi|$) induces an effective advective field $\mathcal{V}_j$. This reveals that anomalous transport in inhomogeneous media can be decomposed into a purely diffusive geometric part and a first-order transport process.
    \item \textbf{Super-diffusive Heavy Tails:} Our asymptotic analysis confirms that the propagator exhibits power-law decay governed by the deformed distance. This confirms that the model captures long-range interactions and Lévy-type jumps, providing a robust mathematical tool for modeling transport in complex metamaterials and graded media.
\end{enumerate}

Ultimately, the use of Fox $H$-functions provides a closed-form representation that preserves the structural richness of the operator, offering a powerful generalization for the study of anomalous phenomena in non-Euclidean geometries.
\section*{Declaration of competing interest}
The author declares that he has no known competing financial interests or personal relationships that could have appeared to influence the work reported in this paper.


\section*{Declaration of Generative AI and AI-assisted technologies in the manuscript preparation process}

During the preparation of this work the author(s) used Gemini (Google) in order to improve the readability and language quality of the manuscript, as well as to assist in structuring the discussion of the physical interpretations. After using this tool/service, the author(s) reviewed and edited the content as needed and take(s) full responsibility for the content of the published article.

\bibliographystyle{elsarticle-num}

\end{document}